\documentclass[11pt]{article}
\usepackage{amsthm, amsmath, amssymb, amsfonts, url, booktabs, tikz, setspace, fancyhdr, bm, mathrsfs}
\usepackage{hyperref}
\usepackage{geometry}
\usepackage{hyperref, enumerate}
\usepackage[shortlabels]{enumitem}
\usepackage[babel]{microtype}
\usepackage[english]{babel}
\usepackage[capitalise]{cleveref}
\usepackage{comment}
\usepackage{bbm}
\usepackage{csquotes}
\usepackage{mathabx}
\usepackage{tikz}
\usepackage{graphicx}
\usepackage{float}
\usepackage{amsmath}

\counterwithin{figure}{section}

\newtheorem{theorem}{Theorem}[section]

\newtheorem{lemma}[theorem]{Lemma}

\newtheorem{conj}[theorem]{Conjecture}
\newtheorem{claim}[theorem]{Claim}

\theoremstyle{definition}

\newtheorem*{defn-non}{Definition}

\usepackage[linesnumbered, ruled]{algorithm2e}
\SetKwRepeat{Do}{do}{while}%

\newenvironment{poc}{\begin{proof}[Proof of the claim]}{\end{proof}}

\usepackage{todonotes}

\newcommand{\cP}{\mathcal{P}}

\newcommand{\cF}{\mathcal{F}}

\newcommand{\VC}{\operatorname{VC}}
\newcommand{\Tr}{\operatorname{Tr}}

\newcommand{\W}{\operatorname{W}}

\title{A disproof of the uniform witness conjecture}
\author{
Zixiang Xu\thanks{School of Mathematical Sciences, Zhejiang University, Hangzhou, China. Email: zixiangxu@zju.edu.cn.}
}
\date{}

\begin{document}
\maketitle

\begin{abstract}
The study of \((d+1)\)-uniform set systems with VC-dimension at most \(d\) links the Erd\H{o}s--Ko--Rado theorem with VC-dimension. But already in 1997, Ahlswede and Khachatrian showed that this is not the right extension of the Erd\H{o}s--Ko--Rado theorem. In 2025, Chao, Xu, Yip and Zhang proposed the uniform witness conjecture as a possible right extension: for \(0\le s\le d\), if every set of a \((d+1)\)-uniform family has a missing trace of the same fixed size \(s\), then the family should have size at most \(\binom{n-1}{d}\). They proved the conjecture when \(s=d\), and when \(s=1\) and \(n\) is large. Very recently, Chao, Xu and Zakharov proved the conjecture when \(s\le \frac{d}{2}\) and \(n\) is large.

We fill in the missing half of the picture, although the picture is not the one suggested by the conjecture. More precisely, for \(d\ge4\) and \(\left\lceil \frac{d+2}{2}\right\rceil\le s\le d-1\), we construct such a family \(\mathcal{F}\subseteq\binom{[n]}{d+1}\) with
\(
|\mathcal{F}|=\binom{n-1}{d}+\binom{n-2(d+1-s)-2}{2s-d-2}
\)
for every \(n\ge2(d+1)\), thereby disproving the uniform witness conjecture.

\end{abstract}

\section{Introduction}

A basic goal in extremal set theory is to find how large or how small a set system can be under a given combinatorial condition. One starting point is the theorem of Erd\H{o}s, Ko and Rado~\cite{1961EKR}, which says that for \(n\ge2k\), every intersecting \(k\)-uniform family on \([n]\) has size at most \(\binom{n-1}{k-1}\). Moreover, when \(n>2k\), equality holds only for a star. This theorem led to many later results on intersection conditions and their variants~\cite{1989Chuan,2017PAMSHan,1967HM,2017HuangZhao,2024PengHuang,1986Pyber}.

Here we focus on a related direction, VC-dimension. Let \(\cF\subseteq 2^X\) be a set system. A set \(S\subseteq X\) is shattered by \(\cF\) if \(\Tr_{\cF}(S)=2^S\), where \(\Tr_{\cF}(S)=\{F\cap S:F\in\cF\}\). The VC-dimension of \(\cF\), denoted by \(\VC(\cF)\), is the largest size of a shattered set. The Sauer--Shelah theorem~\cite{1972JCTASauer,1972PACJMShelah,1971TPAVCVC} gives exactly the maximum size of a not necessarily uniform set system with VC-dimension at most \(d\), namely \(\sum_{i=0}^d\binom{n}{i}\). The Hamming ball of radius \(d\) shows that this bound is sharp.

The uniform version is much less understood. For a \((d+1)\)-uniform family \(\cF\subseteq\binom{[n]}{d+1}\), the condition \(\VC(\cF)\le d\) has a simple form: for every \(F\in\cF\), there is a proper subset \(B_F\subsetneq F\) such that \(F\cap F'\ne B_F\) for every \(F'\in\cF\). Thus every edge misses at least one trace on itself. If \(\cF\) is intersecting, then one may take \(B_F=\emptyset\) for every \(F\), so the Erd\H{o}s--Ko--Rado theorem gives examples of size \(\binom{n-1}{d}\).

Because of this link, Erd\H{o}s~\cite{1984Erdos} and Frankl and Pach~\cite{1984Franklpach} conjectured in the 1980s that \(\binom{n-1}{d}\) should be the maximum size of a \((d+1)\)-uniform family with VC-dimension at most \(d\), when \(n\) is sufficiently large compared to \(d\). Frankl and Pach~\cite{1984Franklpach} proved the general upper bound \(\binom{n}{d}\). However, Ahlswede and Khachatrian~\cite{1997CombFan} disproved the conjecture by constructing a family of size \(\binom{n-1}{d}+\binom{n-4}{d-2}\). Mubayi and Zhao~\cite{2007JAC} later gave infinitely many non-isomorphic constructions with the same size, and conjectured that this Ahlswede--Khachatrian value is the correct answer for large \(n\). More recently, Tran and Xu~\cite{2026TranXu}, and Zhao and Ge~\cite{2026GeZhao} independently further improved the lower bound for \(d\ge 3\), thereby disproving the above conjecture of Mubayi and Zhao.

The upper bound side has also seen several developments. Mubayi and Zhao~\cite{2007JAC} improved the Frankl--Pach bound by an \(\Omega_d(\log n)\) term when \(d\) is a prime power and \(n\) is large. Ge, Xu, Yip, Zhang and Zhao~\cite{2024FranklPach} proved that the Frankl--Pach bound is not tight for any \(d\ge2\) and \(n\ge 2d+2\), giving an unconditional improvement. Chao, Xu, Yip and Zhang~\cite{2025CombProof} later gave a purely combinatorial proof of the stronger estimate \(\binom{n-1}{d}+O_d(n^{d-1-\frac{1}{4d-2}})\). Later, Yang and Yu~\cite{2025YangYu} improved this further to \(\binom{n-1}{d}+O_d(n^{d-2})\) for fixed \(d\) and large \(n\). For \(d=2\), Wang, Xu and Zhang~\cite{2025ThreeUniform} determined the exact value for \(n\ge 7\), confirming the Ahlswede--Khachatrian value in the 3-uniform case. These results show both the strength and the difficulty of the original Erd\H{o}s--Frankl--Pach problem~\cite{1984Erdos,1984Franklpach}. They also show that bounded VC-dimension alone is not the right way to recover the Erd\H{o}s--Ko--Rado theorem.

Chao, Xu, Yip and Zhang~\cite{2025CombProof} proposed a more rigid way to recover the Erd\H{o}s--Ko--Rado theorem. Instead of only asking that every edge has some missing trace, they require all missing traces to have the same fixed size. Let \(0\le s\le d\). A family \(\cF\subseteq\binom{[n]}{d+1}\) is called an \(s\)-witness family if for every \(F\in\cF\), there exists \(B_F\in\binom{F}{s}\) such that \(F\cap F'\ne B_F\) for every \(F'\in\cF\). We write \(\W_{d,s}(n)\) for the maximum size of an \(s\)-witness family in \(\binom{[n]}{d+1}\).

The case \(s=0\) is exactly the Erd\H{o}s-Ko-Rado Theorem. The star is an \(s\)-witness family for every \(s\le d\), since for an edge \(F\) containing the center \(1\), any \(s\)-subset of \(F\setminus\{1\}\) is missing from the trace on \(F\). This led to the following conjecture.

\begin{conj}[Uniform witness conjecture~\cite{2025CombProof}]\label{conj:witness}
Let \(n\ge2(d+1)\) and \(0\le s\le d\). If \(\cF\subseteq\binom{[n]}{d+1}\) is an \(s\)-witness family, then \(|\cF|\le\binom{n-1}{d}\).
\end{conj}

Chao, Xu, Yip and Zhang~\cite{2025CombProof} verified Conjecture~\ref{conj:witness} when \(s=d\), and also when \(s=1\) and \(n\) is sufficiently large. In particular, their result implies the case \(d=2\) for large \(n\). Very recently, Chao, Xu and Zakharov~\cite{2026Chao} proved the conjecture when \(s\le \frac{d}{2}\) and \(n\) is sufficiently large. They also constructed various non-star \(s\)-witness families of size exactly \(\binom{n-1}{d}\) for \(\frac{d}{2}<s\le d-1\), showing that equality already looks quite different beyond the threshold \(\frac{d}{2}\).

Our main result shows that Conjecture~\ref{conj:witness} is false in general.

\begin{theorem}\label{thm:main}
Let \(d\ge4\) and let \(s\) be an integer with
\(\left\lceil \frac{d+2}{2}\right\rceil\le s\le d-1.\) If \(n\ge2(d+1)\), then
\[
\W_{d,s}(n)\ge\binom{n-1}{d}+\binom{n-2(d+1-s)-2}{2s-d-2}.
\]
\end{theorem}
This covers every possible \(s\) when \(d\) is even, and every possible \(s\) except the lower boundary \(s=\frac{d+1}{2}\) when \(d\) is odd. Moreover, combining Theorem~\ref{thm:main} and the results in~\cite{2025CombProof,2026Chao,1961EKR} gives a rather strange picture: the conjecture is true for \(0\le s\le \frac{d}{2}\) and large \(n\), and it is also true for \(s=d\), but the middle parameters violate the conjectured bound.

\section{Proof of Theorem~\ref{thm:main}}

Let \(C\) and \(U\) be disjoint subsets of \([n]\), let \(\cP\subseteq 2^C\), and define
\[
\cF(\cP,U)=\bigg\{P\cup A:P\in\cP,\ A\in\binom{U}{d+1-|P|}\bigg\}.
\]
The set \(C\) is fixed and small. A set \(P\in\cP\) records the part of an edge inside \(C\), while \(A\subseteq U\) fills the remaining vertices.

\begin{lemma}\label{lem:criterion}
Let \(r=d+1-s\). Suppose that for every \(P\in\cP\), there is a set \(\tau(P)\subsetneq P\) such that \(|P|-|\tau(P)|\le r\), \(|\tau(P)|\le s\), and there is no \(Q\in\cP\) satisfying
\(
Q\cap P=\tau(P)\) and \(|Q|\le |\tau(P)|+r.\)
Then \(\cF(\cP,U)\) is an \(s\)-witness family.
\end{lemma}

\begin{proof}[Proof of Lemma~\ref{lem:criterion}]
For every \(P\in\cP\), we have \(|P|\le |\tau(P)|+r\le s+r=d+1\), so the family above is well-defined. Fix \(F=P\cup A\in\cF(\cP,U)\). Since \(|P|-|\tau(P)|\le r=d+1-s\), we get
\[
s-|\tau(P)|\le d+1-|P|=|A|.
\]
Choose \(X\subseteq A\) with \(|X|=s-|\tau(P)|\), and put \(B_F=\tau(P)\cup X\). Then \(B_F\subseteq F\) and \(|B_F|=s\). We claim that \(B_F\) is missing from the trace on \(F\). Indeed, suppose that \(F'=Q\cup A'\in\cF(\cP,U)\) satisfies \(F\cap F'=B_F\). Comparing the vertices in \(C\) and in \(U\) separately gives
\(Q\cap P=\tau(P)\) and \(A'\cap A=X.\)
Hence
\[
d+1-|Q|=|A'|\ge |X|=s-|\tau(P)|,
\]
which further implies that \(|Q|\le |\tau(P)|+r\), contradicting the assumption on \(\tau(P)\). Thus \(B_F\notin\Tr_{\cF}(F)\), and \(\cF(\cP,U)\) is an \(s\)-witness family.
\end{proof}

Put \(r=d+1-s\). From the assumptions we have
\(r\ge2,\) \(s\ge r+1\) and \(d-2r=2s-d-2\ge 0.\)
In particular \(2r\le d\), so there is enough room to choose two disjoint \(r\)-sets \(T,L\subseteq[n]\setminus\{1,2\}\). Now fix
\(
C:=\{1,2\}\cup T\cup L\) and \(U=[n]\setminus C.\)
Define
\[
\cP=\left(\{P\subseteq C:1\in P,\ |P|\le d+1\}\setminus\{\{1\}\cup T\cup Y:Y\subsetneq L\}\right)\cup\{\{2\}\cup T\cup Y:Y\subseteq L\}.
\]
So we start from the \(1\)-star inside \(C\), remove all sets \(\{1\}\cup T\cup Y\) with \(Y\subsetneq L\), put back the corresponding sets \(\{2\}\cup T\cup Y\), and also add the last set \(\{2\}\cup T\cup L\). This last set has size \(1+2r\le d+1\), so it is included in the construction.

\begin{claim}\label{claim:main-count}
For \(\cF=\cF(\cP,U)\),
\(
|\cF|=\binom{n-1}{d}+\binom{n-2r-2}{d-2r}.
\)
\end{claim}

\begin{poc}
The full \(1\)-star inside \(C\), namely \(\{P\subseteq C:1\in P,\ |P|\le d+1\}\), gives exactly all \((d+1)\)-sets containing \(1\). Hence it contributes \(\binom{n-1}{d}\). For every \(Y\subsetneq L\), the removed set \(\{1\}\cup T\cup Y\) and the added set \(\{2\}\cup T\cup Y\) have the same size, so they contribute the same number of edges. The only added set without a removed partner is \(\{2\}\cup T\cup L\). Its size is \(2r+1\), and its contribution is \(\binom{|U|}{d+1-(2r+1)}=\binom{n-2r-2}{d-2r}.\)
This proves the claim.
\end{poc}

\begin{claim}\label{claim:main-certificate}
The family \(\cP\) satisfies the assumptions of Lemma~\ref{lem:criterion}.
\end{claim}

\begin{poc}
The main task is to choose \(\tau(P)\) for each \(P\in\cP\).

First take \(P=\{2\}\cup T\cup Y\) with \(Y\subsetneq L\). Set \(\tau(P)=T\). Then \(\tau(P)\subsetneq P\), \(|P|-|\tau(P)|=1+|Y|\le r\), and \(|\tau(P)|=r\le s\). Suppose that some \(Q\in\cP\) satisfies \(Q\cap P=T\). If \(Q\) is also of the form \(\{2\}\cup T\cup Z\), then \(2\in Q\cap P\), impossible. Thus \(Q\) must contain \(1\). To have intersection exactly \(T\), it must be of the form \(\{1\}\cup T\cup Z\), where \(Z\subseteq L\) and \(Z\cap Y=\emptyset\). If \(Z\subsetneq L\), then this set was removed from \(\cP\). If \(Z=L\), then \(Y=\emptyset\), and in this case \(|Q|=1+2r>2r=|\tau(P)|+r.\)
So no such \(Q\) satisfies the size condition in Lemma~\ref{lem:criterion}.

Next take the remaining added set \(P=\{2\}\cup T\cup L\). Fix one element \(\ell_0\in L\), and set \(\tau(P)=T\cup\{\ell_0\}\). Then \(\tau(P)\subsetneq P\), \(|P|-|\tau(P)|=r\), and \(|\tau(P)|=r+1\le s\). A set of the form \(\{2\}\cup T\cup Z\) cannot have intersection \(\tau(P)\) with \(P\), because the intersection contains \(2\). A set containing \(1\) would have to contain \(T\cup\{\ell_0\}\), avoid \(2\), and avoid every point of \(L\setminus\{\ell_0\}\). Hence it would be exactly \(\{1\}\cup T\cup\{\ell_0\}\), which was removed. Thus the required exclusion also holds for this \(P\).

It remains to consider \(P\in\cP\) with \(1\in P\). If \(|P|\le r\), set \(\tau(P)=\emptyset\). Then \(\tau(P)\subsetneq P\), \(|P|-|\tau(P)|\le r\), and \(|\tau(P)|\le s\). Notice that no set containing \(1\) can have empty intersection with \(P\), since \(1\in P\). Any set of the form \(\{2\}\cup T\cup Z\) has size at least \(r+1\), so it cannot satisfy \(|Q|\le r=|\tau(P)|+r\).

Now assume that \(1\in P\), \(|P|>r\), and \(P\cap T=\emptyset\). Then all points of \(P\), except possibly \(1\) and \(2\), lie in \(L\). Since \(r\ge2\), there are enough points in \(P\cap L\) to choose
\(\tau(P)\subseteq P\cap L\) with \(|\tau(P)|=|P|-r.\)
Then \(\tau(P)\) does not contain \(1\). Hence no pattern containing \(1\) can have intersection \(\tau(P)\) with \(P\). On the other hand, if \(Q\) is one of the added sets, say \(Q=\{2\}\cup T\cup Z\), and \(Q\cap P=\tau(P)\), then \(Q\) must contain \(T\) and all points of \(\tau(P)\). Thus \(|Q|\ge |\tau(P)|+r+1,\)
which yields that this set \(Q\) is too large to be relevant in Lemma~\ref{lem:criterion}.

Finally assume that \(1\in P\), \(|P|>r\), and \(P\cap T\ne\emptyset\). Since \(r\ge 2,\) we can choose \(t_0\in P\cap T\), and further choose \(\tau(P)\subseteq P\setminus\{1,t_0\}\) with \(|\tau(P)|=|P|-r.\) Clearly, \(\tau(P)\subsetneq P\) and \(|\tau(P)|\le s\). Notice that any set containing \(1\) has intersection with \(P\) containing \(1\), and any set of the form \(\{2\}\cup T\cup Z\) has intersection with \(P\) containing \(t_0\). Since \(\tau(P)\) contains neither \(1\) nor \(t_0\), no such \(Q\) can have \(Q\cap P=\tau(P)\). This completes the verification and proves the claim.
\end{poc}

By Lemma~\ref{lem:criterion} and Claims~\ref{claim:main-count} and~\ref{claim:main-certificate}, the family \(\cF\) is an \(s\)-witness family and
\[
|\cF|=\binom{n-1}{d}+\binom{n-2(d+1-s)-2}{2s-d-2}.
\]
The second term is positive, because \(d-2r\ge0\) and \(n-2r-2\ge d-2r\). This proves Theorem~\ref{thm:main}.
\section{Some remarks}
Together with the known results for \(s\le \frac{d}{2}\)~\cite{2026Chao} and for
\(s=d\)~\cite{2025CombProof}, Theorem~\ref{thm:main} leaves only a very thin
problem whether
\(
\W_{2s-1,s}(n)>\binom{n-1}{2s-1}
\)
can hold for some fixed \(s\ge2\) and infinitely many \(n\).

More generally, although Conjecture~\ref{conj:witness} fails in the range
covered by Theorem~\ref{thm:main}, the star may still give the correct main
term. Our construction has excess
\[
\binom{n-2(d+1-s)-2}{2s-d-2}
=\Theta_{d,s}\bigl(n^{2s-d-2}\bigr).
\]
It would be interesting to determine whether this order of magnitude in error term is
optimal.

\section*{Acknowledgement}
Zixiang Xu would like to thank Ting-Wei Chao for discussions on this problem at IBS during the winter of 2024, Dmitrii Zakharov for kindly sharing and discussing the proof of the case \(s\le \frac{d}{2}\) in~\cite{2026Chao}, and Prof. Jian Wang and Jialuo Wang for several discussions on attempts to prove this conjecture for small parameters.

\bibliographystyle{abbrv}
\bibliography{uniform_witness_disproof}

\end{document}